\documentclass[ECP]{ejpecp}
\usepackage{tikz,color}
\usepackage[labelfont={up,bf}, font=it]{caption}
\usetikzlibrary{arrows,shapes,decorations,backgrounds,chains}
\usepackage{url}

\newcommand{\arXiv}[1]{\href{http://arxiv.org/abs/#1}{arXiv:#1}}

\newcommand{\ep}{\varepsilon}

\renewcommand{\Pr}{\mathbb{P}}
\newcommand{\eps}{\varepsilon}

\renewcommand{\hat}{\widehat}

\DeclareMathOperator{\bern}{Bernoulli}
\newcommand{\df}{\textbf}

\SHORTTITLE{Avoidance Coupling}
\TITLE{Avoidance Coupling}
\AMSSUBJ{60J10}
\SUBMITTED{August 26, 2012}
\ACCEPTED{July 3, 2013}
\VOLUME{18}
\YEAR{2013}
\PAPERNUM{58}
\DOI{v18-2275}

\AUTHORS{
\begin{tabular}{c}
\href{http://www.math.ubc.ca/~angel/}{Omer Angel}\\[-3pt]
\small University of British Columbia
\end{tabular}
\and
\begin{tabular}{c}
\href{http://research.microsoft.com/~holroyd/}{Alexander E. Holroyd}\\[-3pt]
\small Microsoft Research
\end{tabular}
\and
\begin{tabular}{c}
\href{http://www.stats.ox.ac.uk/~martin/}{James Martin}\\[-3pt]
\small University of Oxford
\end{tabular}
\and
\begin{tabular}{c}
\href{http://dbwilson.com}{David B. Wilson}\\[-3pt]
\small Microsoft Research
\end{tabular}
\and
\begin{tabular}{c}
\href{http://www.math.dartmouth.edu/~pw/}{Peter Winkler}\\[-3pt]
\small Dartmouth College
\end{tabular}
}
\ABSTRACT{
  We examine the question of whether a collection of random walks on a
  graph can be coupled so that they never collide.  In particular, we
  show that on the complete graph on $n$ vertices, with or without
  loops, there is a Markovian coupling keeping apart $\Omega(n/\log
  n)$ random walks, taking turns to move in discrete time.
}

\begin{document}

\section{Introduction}

Coupling of Markov chains has proved to be a valuable tool, notably,
in recent years, in proving rapid mixing.  Our intent here is to isolate
one very simple type of Markov chain (random walk, especially on a complete
graph) and to explore one particular capability, that of avoiding collision.

As an application, one may envisage some anti-virus software
moving from port to port in a computer system to check for incursions.  It is
natural to have such a program implement a random walk on the ports so as
not to be predictable.  If another program (possibly with a different
purpose) also does a random walk on the ports, it may be desirable or even
essential to prevent the programs from examining the same port at the same time.

If two random walks are independent, they will collide in polynomial time
on any finite, connected, non-bipartite graph, even if a scheduler tries to keep them apart \cite{CTW,TW}.
Only if the scheduler is clairvoyant---that is, knows the entire future of each
walk---is there a possibility of avoiding collision forever, and that case rests
on a complex proof \cite{G} for enormous graphs.

Coupling, on the other hand, is a much more powerful technique for
keeping random walks apart.  On the cycle $C_n$, for example (where the clairvoyant
scheduler has no chance), coupling can easily keep linearly many random walks
apart, simply by having them either all move clockwise, or all counter-clockwise, at
the same time.

Keeping random walks apart on a complete graph $K_n$ by coupling---especially
Markovian coupling, which we define more formally below---appears to be a
more difficult task. We apply a number of techniques to achieve such
couplings, depending on number-theoretic properties of $n$.  For infinitely
many $n$ there is a Markovian coupling which keeps apart $(n-1)/2$ random
walkers, and for all $n$ there is a Markovian coupling which keeps apart
$\Omega(n/\log n)$ random walkers.  We have essentially the same results on
the looped version of the complete graph $K_n^*$, and in this case we also have
non-Markovian couplings of linearly many walkers, for all $n$.

The closely related problem of coupling two Brownian motions (on various
domains) so as to keep them at least some positive distance apart has been
studied in some depth---see \cite{benjamini-burdzy-chen,kendall}, and
\cite{bramson-burdzy-kendall} for a recent connection to pursuit-evasion
problems.

\section{Preliminaries}

We refer the reader to a modern text such as \cite{AF,LPW} (both of which are
accessible online), for background on discrete Markov chains.  All of our
Markov chains are time-homogeneous and have finite state spaces.

A \df{coupling} of Markov chains is nothing more than an implementation of
the chains on a common probability space, in such a way that each chain,
viewed separately, is faithful to its transition matrix. In what follows,
$X_t$ and $Y_t$ for $t = 0,1,2,\dots$ will represent simple discrete-time
random walks on the loopless complete graph $K_n$, or its looped counterpart
$K_n^*$ (in which each vertex has a single self-loop, and the walk stays
where it is with probability $1/n$ at each step). Clearly a time~$t$ for
which $X_t = Y_t$ should constitute a collision, but what if $X_{t+1} = Y_t$?
If in addition $Y_{t+1} = X_t$ we call such an event a ``swap'', otherwise a
``shove''.

Allowing swaps and shoves makes things easy---on $K_n$, for example, we could
couple $n$ walks simply by choosing, at each turn, a uniformly random derangement
(or a uniformly random cycle)~$\sigma$, and having the walker at $i$ move to $\sigma(i)$.

Instead, we make the issue of swaps and shoves moot by having the walkers
move alternately---in the case of two walkers, in the order $X_0, Y_0, X_1,
Y_1,\dots$.  Then the events $X_t = Y_t$ and $X_{t+1}=Y_t$ both constitute
collisions.  Multiple walkers are assumed to take turns in a fixed cyclic
order, and again, a collision is deemed to occur exactly if a walker moves to
a vertex currently occupied by another. We call a coupling that forbids
collisions an \df{avoidance coupling}.

Note that a collection of random walkers who move in \textit{continuous\/}
time, that is, after independent exponential (mean 1) waiting times, can be
coupled so as to take turns as above; if there are $k$ walkers, we
simply have walker $j$ wait for a random time after walker $j{-}1$ (modulo $k$)
has moved, where the time is distributed according to a Gamma distribution
with shape parameter $1/k$ and scale parameter~$1$.  (The sum of $k$
independent such random variables is an exponential random variable with mean~1.)
Thus, any coupling of our alternating discrete-time walkers can be
applied to the continuous-time case.

A coupling is \df{Markovian} if it is itself a Markov chain, meaning, in the
two-walker case, that $X_{t+1}$ depends only on $X_t$ and $Y_t$, while
$Y_{t+1}$ depends only on $Y_t$ and $X_{t+1}$.  (Brownian couplings with the
analogous property are referred to as ``co-adapted'' in \cite{kendall}.)  To
allow the walkers to alternate, we tacitly assume that the state of the
coupled chain includes the information of whether it is the first or second
player's turn to move.  For multiple walkers, the dependence is, similarly,
only on the current locations of all the walkers, and on whose turn it is to
move.

In the couplings that we construct, the individual walkers may be taken to be
stationary, in the sense that their initial states $X_0$ and $Y_0$ are each
uniformly distributed over the $n$ states.  However, we permit $X_0$ and
$Y_0$ to be coupled in an arbitrary way.  Any such coupling may be modified
to make $(X_0,Y_0)$ uniformly random over all pairs (perhaps at the expense
of the Markovian property) by applying a random permutation to the states.

Note that the transformation above from discrete-time to continuous-time chains
does not preserve the Markovian property. Indeed, it is rarely possible to
get a Markovian avoidance coupling for continuous-time chains:
\begin{theorem}
Let $M$ be an irreducible, continuous-time, finite-state Markov chain. Then
there is no Markovian coupling of two or more copies of $M,$ without
simultaneous transitions, that avoids all collisions.
\end{theorem}

\begin{proof}
  Suppose $X$ and $Y$ are copies of $M$ that are coupled in this way.
  Since $M$ is irreducible, we may fix some tour $v_0,v_1,\dots,v_k$ of all
  the states that has a positive probability.  Let $r_i>0$ be the rate at
  which the transition $v_{i-1} \to v_i$ occurs, and consider the
  probability $p$ that when started in state $v_0$, the single chain $X$
  follows the tour exactly and completes it in time less than $\ep$. Then
\[
p = (1 + {\rm o}(1)) \prod_{i=1}^k r_i \frac{\ep^k}{k!} = \Theta(\ep^k),
\]
where the constants implied by the $\Theta$ notation depend on the Markov chain
and the tour, but not on $\ep$.

Next we start the coupled chain $(X,Y)$ in state $(v_0,s)$ for some $s$, and
consider the probability $q(s)$ that its projection onto the first chain
takes the tour and completes it in time less than~$\ep$.  Since the coupled
chain is collision-avoiding, it must take at least one additional step in
order to move the second walker out of the way.  But then at least $k+1$
transitions must take place within time $\ep$, thus \[q(s) \le
\sum_{j=k+1}^\infty \frac{\ep^{j} R^{j}}{j!} = {\rm O}(\ep^{k+1})\] where
$R$ is the maximum, over all states of the coupled chain, of the rate of
transition out of that state.

For the coupling to be faithful, however, we must have $q(s) \ge p$ for some
$s$.  Since $p = \Theta(\ep^k)$ and $q(s) = {\rm O}(\ep^{k+1})$, this is
impossible for small enough $\ep$.
\end{proof}

\section{Two walkers on three vertices}

No avoidance coupling is possible for two walkers on $K_3$, since there is no
choice of where to move, hence no room for randomness.  On the looped graph
$K_3^*$, however, a walker stays where she is with probability $1/3$.  We
shall see that this is enough to permit an avoidance coupling, but not a
Markovian one. In fact, we can completely analyze the more general walk on
$K_3^*$ in which a walker stays in place with some arbitrary probability $s$,
and moves to each of the other two vertices with probability $(1{-}s)/2$.

\begin{theorem} \label{k3loops}
Consider two walkers on $K_3^*$, each with looping probability $s\in[0,1)$.
There exists an avoidance coupling if and only if $s\geq \tfrac 13$, and
there exists a Markovian avoidance coupling if and only if $s\geq \tfrac12$.
\end{theorem}

\begin{proof}[Proof of Theorem~\ref{k3loops}, non-Markovian case]
We first show that an avoidance coupling exists in the case $s=1/3$ (i.e.,
ordinary random walk on $K_3^*$).  We start the coupled chain in a uniformly
random pair of states $(X_0,Y_0)$ such that $X_0\neq Y_0$.
Given $X_t$ and $Y_t$, the pair
$(X_{t+1},Y_{t+1})$ is chosen uniformly at random among the allowed pairs,
\textit{except\/} $(X_t,Y_t)$ itself (see Figure~\ref{fig:32}).

Thus, for example, if Alice is at 0 and Bob at 1, their new positions will be
$(0,2)$, $(2,0)$, or $(2,1)$ each with probability $1/3$.  Notice that this
coupling is not quite Markovian, as Bob's move depends on Alice's
\textit{previous\/} position---he is not permitted to stay put when Alice has
just done so.

We prove by induction on $t$ that $Y_t$ is uniformly random $\neq X_t$,
independent of $X_s$ and~$Y_s$ for $s<t$.  We may assume $X_t = 0$ for the
purpose of showing that $Y_{t+1}$ is uniform $\neq X_{t+1}$ given $X_{t+1}$;
then, using the induction hypothesis and the coupling definition, the triple
$(Y_t, X_{t+1}, Y_{t+1})$ is equally likely to be any of
$(1,0,2)$, $(1,2,0)$, $(1,2,1)$, $(2,0,1)$, $(2,1,0)$, $(2,1,2)$, which completes the
induction.

Using this fact it easily follows that Alice's sequence is i.i.d.\ uniform;
using the fact again, it follows that Bob's sequence is also i.i.d.\ uniform,
as required.  (We remark that this coupling is invariant under time reversal,
except that Alice and Bob exchange roles.)
\tikzset{vertex/.style = {fill, circle, inner sep = 1pt}}

\begin{figure}[t]
\begin{center}
\begin{tikzpicture}[scale=1.5]
\node (01) at (1,0) {0,1};
\node (02) at (0.5,0.866) {0,2};
\node (12) at (-0.5,0.866) {1,2};
\node (10) at (-1,0) {1,0};
\node (20) at (-0.5,-0.866) {2,0};
\node (21) at (0.5,-0.866) {2,1};
\draw[<->, very thick](01)--(02);
\draw[<->, very thick](10)--(12);
\draw[<->, very thick](21)--(20);
\draw[<->, very thick](10)--(20);
\draw[<->, very thick](01)--(21);
\draw[<->,very thick](12)--(02);
\draw[->, very thick](01)--(20);
\draw[->, very thick](02)--(10);
\draw[->, very thick](10)--(21);
\draw[->, very thick](12)--(01);
\draw[->, very thick](21)--(02);
\draw[->, very thick](20)--(12);
\end{tikzpicture}
\end{center}
\begin{center}
$\cdots,\ \ 1,2,\ \ 0,2,\ \ 0,1,\ \ 2,1,\ \ 0,2,\ \ 1,2,\ \ 0,1,\ \ 0,2,\ \ \cdots$
\end{center}
\caption{
Illustration of the avoidance coupling on $K^*_3$ for $s=1/3$.  The states of Alice and Bob are naturally grouped into pairs, so that Alice and Bob are effectively jointly walking on these pairs (state diagram on top).  If the sequence of pairs (bottom) is reversed, and the elements in each pair is reversed, then the law of this new sequence of pairs is the same as for the original sequence.
}
\label{fig:32}
\end{figure}

Turning now to the case $s\geq 1/3$, we can modify the above coupling as
follows. At each round, with a suitable probability let both walkers stay in
place. Otherwise they proceed to the next round.  This clearly increases the
probability that each walker stays in place at any step, without otherwise
changing their trajectories.

Finally we must show that no avoidance coupling is possible if $s<1/3$.
Consider the event that $X_0,\dots,X_n$ alternate between two (unspecified)
states of $K_3^*$. This has probability $2\,(\tfrac{1-s}{2})^n$, since each
jump has probability $(1-s)/2$.
However, this event forces
$Y_0=Y_1=\dots=Y_{n-1}$, which has probability $s^{n-1}$.  Thus
$2\,(\tfrac{1-s}{2})^n \leq s^{n-1}$.
Taking $n$th roots and letting $n\to\infty$ we find $(1-s)/2 \leq s$, so $s\geq 1/3$.
\end{proof}

\begin{proof}[Proof of Theorem~\ref{k3loops}, Markovian case]
Suppose first that there is a Markovian avoidance coupling. Let $p_{ab}$ be
the probability that Alice stays at $a$ given that it is her move, that she
is at $a$, and that Bob is at~$b$.  Let $q_{ab}$ be the probability that Bob
stays at $b$, given that it is Bob's move, and again that Alice is at~$a$ and
Bob at $b$.  That these quantities may only be defined for certain pairs
$a,b$ will not interfere with our arguments.

Suppose Alice has just moved from $0$ to $1$. Her conditional probability of
next moving back to~$0$ is $(1{-}s)/2$.  Bob must have been at $2$ and will
stay there with probability $q_{12}$, after which Alice moves to~$0$ with
probability $1{-}p_{12}$.   We conclude that $(1{-}s)/2 = q_{12}(1{-}p_{12})$.

Similarly, suppose Bob has just moved from $0$ to $2$.  His conditional
probability of next moving back to $0$ is $(1{-}s)/2$.  Alice must have been at
$1$ and will stay there with probability $p_{12}$, after which Bob moves to
$0$ with probability $1{-}q_{12}$.  So we get $(1{-}s)/2 = p_{12}(1{-}q_{12})$.
Combined with the conclusion of the previous paragraph, this gives $p_{12} =
q_{12}$, and similarly $p_{ab} = q_{ab}$ for all $a \neq b$.

Since the equation $p_{12}(1{-}p_{12}) = (1{-}s)/2$ has no real solutions for $s
< \frac12$, the presumed coupling cannot exist in this case.

We now demonstrate that, conversely, when $\frac12 \le s < 1$ there is a
Markovian avoidance coupling for two walkers.  Let $p$ and $1{-}p$ be the two
(possibly equal) values of $x$ satisfying $x(1-x)=(1-s)/2$, and note that then $p^2
+ (1{-}p)^2 = s$. Letting $i'$ stand for $i{+}1 \mod 3$, put $p_{ii'} = p$,
$p_{i'i} = 1{-}p$, and $q_{ij}=p_{ij}$.  We claim that these values are the holding probabilities
$p_{ij},q_{ij}$ (as defined earlier in the proof) of a Markovian avoidance coupling.

To show this, condition on the event that Alice is at $i$ at time $1$.  We
will show that, conditioned also on Bob's position at time $0$, Alice's next
step is to $i$ (respectively $i'$) with the correct probability~$s$
(respectively $(s{-}1)/2$); hence the probability she moves to $i''$ is
correct also.  Since the coupling is Markovian, Alice's future depends on her
past only through $(X_1,Y_0)$, so this will suffice to prove that Alice's
trajectory has the correct law.  By the symmetry of our construction, the
same will then apply to Bob.

Suppose first that $Y_0=i'$, that is, that Bob was at $i'$ one move ago.
Then, with all probabilities conditional on $\{X_1=i, Y_0=i'\}$,
\begin{align*}
\Pr(X_2=i)&=\Pr(Y_1=i')\Pr(X_2=i \mid Y_1=i') + \Pr(Y_1=i'')\Pr(X_2=i \mid Y_1=i'')\\
&=q_{i i'}p_{i i'} + (1{-}q_{i i'})p_{i i''} = p^2 + (1{-}p)^2 = s,
\intertext{and}
\Pr(X_2=i')&=\Pr(Y_1=i'')\Pr(X_2=i' \mid Y_1=i'') = (1{-}q_{i i'})(1{-}p_{i i''}) = (1{-}p)p=\tfrac{1-s}{2}.
\end{align*}
Observe that the coupling is invariant under replacing state $i$ with $-i\bmod 3$, swapping $'$ and $''$, and substituting $1-p$ for $p$.  Since the above conditional probabilities are symmetric in $p$ and $1-p$, it follows that the distribution of $X_2$ conditional on $\{X_1=i, Y_0=i''\}$ is also correct.
\end{proof}

\section{Two walkers for composite \texorpdfstring{\boldmath$n$}{n}}

\begin{figure}[b]
\begin{center}
\tikzset{vertex/.style = {fill, circle, inner sep = 1pt}}
\begin{tikzpicture}
\begin{scope}[shift={(0,0)},scale=0.34]
\foreach\x in{0,1,2} \foreach\y in{0,1,2,3} \node (/x,/y) at (\x,\y) [vertex]{};
\node at (2,1) [vertex,white] {};
\node at (2,1) {A};
\node at (1,3) [vertex,white] {};
\node at (1,3) {B};
\end{scope}
\begin{scope}[shift={(1.63125,0)},scale=0.34]
\foreach\x in{0,1,2} \foreach\y in{0,1,2,3} \node (/x,/y) at (\x,\y) [vertex]{};
\node at (1,2) [vertex,white] {};
\node at (1,2) {A};
\node at (1,3) [vertex,white] {};
\node at (1,3) {B};
\end{scope}
\begin{scope}[shift={(3.2625,0)},scale=0.34]
\foreach\x in{0,1,2} \foreach\y in{0,1,2,3} \node (/x,/y) at (\x,\y) [vertex]{};
\node at (1,2) [vertex,white] {};
\node at (1,2) {A};
\node at (0,1) [vertex,white] {};
\node at (0,1) {B};
\end{scope}
\begin{scope}[shift={(4.89375,0)},scale=0.34]
\foreach\x in{0,1,2} \foreach\y in{0,1,2,3} \node (/x,/y) at (\x,\y) [vertex]{};
\node at (0,3) [vertex,white] {};
\node at (0,3) {A};
\node at (0,1) [vertex,white] {};
\node at (0,1) {B};
\end{scope}
\begin{scope}[shift={(6.525,0)},scale=0.34]
\foreach\x in{0,1,2} \foreach\y in{0,1,2,3} \node (/x,/y) at (\x,\y) [vertex]{};
\node at (0,3) [vertex,white] {};
\node at (0,3) {A};
\node at (2,2) [vertex,white] {};
\node at (2,2) {B};
\end{scope}
\begin{scope}[shift={(8.15625,0)},scale=0.34]
\foreach\x in{0,1,2} \foreach\y in{0,1,2,3} \node (/x,/y) at (\x,\y) [vertex]{};
\node at (0,1) [vertex,white] {};
\node at (0,1) {A};
\node at (2,2) [vertex,white] {};
\node at (2,2) {B};
\end{scope}
\begin{scope}[shift={(9.7875,0)},scale=0.34]
\foreach\x in{0,1,2} \foreach\y in{0,1,2,3} \node (/x,/y) at (\x,\y) [vertex]{};
\node at (0,1) [vertex,white] {};
\node at (0,1) {A};
\node at (2,0) [vertex,white] {};
\node at (2,0) {B};
\end{scope}
\begin{scope}[shift={(11.4188,0)},scale=0.34]
\foreach\x in{0,1,2} \foreach\y in{0,1,2,3} \node (/x,/y) at (\x,\y) [vertex]{};
\node at (2,2) [vertex,white] {};
\node at (2,2) {A};
\node at (2,0) [vertex,white] {};
\node at (2,0) {B};
\end{scope}
\begin{scope}[shift={(13.05,0)},scale=0.34]
\foreach\x in{0,1,2} \foreach\y in{0,1,2,3} \node (/x,/y) at (\x,\y) [vertex]{};
\node at (2,2) [vertex,white] {};
\node at (2,2) {A};
\node at (1,0) [vertex,white] {};
\node at (1,0) {B};
\end{scope}
\end{tikzpicture}
\end{center}
\caption{
Illustration of the avoidance coupling for two walkers on $K_n$ for composite~$n$ in the case $n=3\times 4$.  The ``clusters'' are the columns.  Alice usually moves to Bob's cluster, but sometimes stays in her own.  Bob moves to a new cluster when Alice is in his cluster, and otherwise stays in his cluster.  The time-reversed process, with the roles of Alice and Bob exchanged, is equal in law to the original process.
}
\label{fig:composite}
\end{figure}
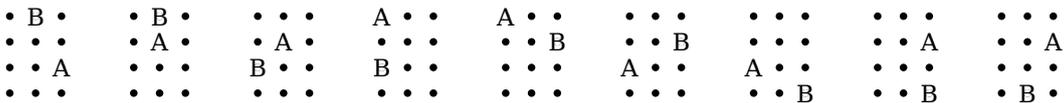

\begin{theorem}\label{thm:ab}
  For any composite $n=ab$, where $a,b>1$, there exist Markovian
  avoidance couplings for two walkers on $K_n$ and on $K_n^*$.
\end{theorem}

\begin{proof}
  We partition $[n] := \{0,1,\dots,n{-}1\}$ into $b$ ``clusters''
  $S_1,\dots,S_b$ each of size $a$.  We construct a coupling so that,
  when it is Alice's turn to move, she and Bob are in different
  clusters.  (This is where we use $b>1$.)

  For the coupling on $K_n$, Alice's protocol is to move with
  probability $\frac{a(b-1)}{ab-1}$ to a random vertex in Bob's cluster
  (other than Bob's vertex), and move with probability
  $\frac{a-1}{ab-1}$ to a random vertex in her own cluster (other than her
  current vertex).  (This is where we use $a>1$.)  Bob's protocol is
  to move to a random new vertex in his current cluster, unless Alice is
  also in his cluster, in which case he moves to a uniformly random vertex
  in a uniformly random unoccupied cluster.  (After Bob moves, he and
  Alice are once again in different clusters.)

  The coupling for $K_n^*$ is essentially the same, except that
  Alice's probability of moving to Bob's cluster is $\frac{b-1}b$, and
  when either Alice or Bob move within their own cluster, the new vertex
  may be the same as the current vertex.

  Regardless of how Alice and Bob start, after Alice and then Bob
  move, Bob is at a uniformly random new vertex (for the coupling on
  $K_n$) or a uniformly random vertex (for the coupling on~$K_n^*$).
  Thus Bob's walk has the correct distribution.

  When the coupled walks are viewed backwards in time, the protocol is
  the same but with the roles of Alice and Bob reversed.  Thus Alice's
  walk also has the correct distribution.
\end{proof}

Note that when $a=2$, the above coupling has \df{minimum entropy}, meaning
that the entropy of the coupling is equal to the entropy of a single walker.
The coupling does not have minimum entropy for $a>2$; it can, however, be
modified to have minimum entropy, at the cost of giving up the Markovian
property.  Specifically, we can give each cluster the structure of a directed
cycle and insist that when Alice moves to Bob's cluster, she chooses the site
after Bob's; and if she stays in her own cluster, Bob copies Alice's movement
in his own cluster. To copy Alice's movement, Bob needs to remember where
Alice was on her previous turn, which is why this modified coupling is not
Markovian.

\section{Monotonicity}

The purpose of this section is to show that existence of avoidance couplings
for $k$ walkers on $K_n^*$ is monotone in $n$.  We do not know whether the
corresponding statement holds for the unlooped case $K_n$, nor if we impose the
Markovian condition or the minimum entropy condition.

\begin{theorem}\label{mono-n}
  If there is an avoidance coupling of $k$ walkers on $K_n^*$, then
  there is an avoidance coupling of $k$ walkers on $K_{n+1}^*$.
\end{theorem}

The following concept will be useful for the proof.  Suppose that $k$ walkers
walk on $K_2^*$, taking turns in cyclic order as usual, in such a way that no
two walkers are simultaneously at vertex $1$ (but several walkers can be at
$2$), and so that the trajectory of any given walker is an i.i.d.\ Bernoulli
sequence that is $1$ with probability $p$ at each step.  We call such a
coupling a \df{1-avoidance} coupling of $\bern(p)$ walkers.

\begin{lemma}\label{1-over-n}
  If there is an avoidance coupling of $k$ walkers on $K_n^*$, then
  there is a 1-avoidance coupling of $k$ $\bern(1/n)$ walkers.
\end{lemma}

\begin{proof}
Let any given Bernoulli walker be at $1$ exactly when the
corresponding walker on $K_n^*$ is at vertex~$1$.
\end{proof}

\begin{lemma}\label{mono-p}
  If there is a 1-avoidance coupling of $k$ $\bern(p)$ walkers, then
  there is a 1-avoidance coupling of $k$ $\bern(q)$ walkers for all
  $q<p$.
\end{lemma}

\begin{proof}
We simply thin the process of $1$s.  Suppose we have a $\bern(p)$ coupling,
and take an independent process of i.i.d.\ coin flips, heads with probability
$q/p$, indexed by turns (i.e.\ times at which any walker is allowed to move).
To get the $\bern(q)$ coupling, take a walker to be at $1$ whenever the
original walker is at~$1$ and the corresponding coin flip is heads.
\end{proof}

\begin{proof}[Proof of Theorem~\ref{mono-n}]
Suppose we have an avoidance coupling of $k$ walkers on
$K_n^*$.  By Lemmas~\ref{1-over-n} and~\ref{mono-p}, there
exists a 1-avoidance coupling of $k$ $\bern(1/(n+1))$
walkers.  Take such a coupling, independent of the original
coupling on $K_n^*$.  To get a coupling on $K_{n+1}^*$, take a
given walker to be at the same vertex as the corresponding
walker on $K_n^*$, unless the corresponding Bernoulli walker is
at~$1$, in which case take it to be at vertex $n+1$.
\end{proof}

\section{Linear number of walkers for special \texorpdfstring{\boldmath$n$}{n}}

\begin{theorem}\label{L:linear}
  There exists a minimum-entropy Markovian avoidance coupling for $k$ walkers on $K_n^*$ for any $k\leq2^d$ and
  any $n=2^{d+1}$ or $2^{d+1}+1$, as well as on $K_n$ for $n=2^{d+1}+1$.
\end{theorem}

The avoidance coupling of $2^d$ walkers on $K_{2^{d+1}+1}$ is illustrated in Figure~\ref{fig:linear}.

\begin{corollary}
There exists an avoidance coupling for $k$ walkers on $K_n^*$ for any $k\le
n/4$.
\end{corollary}

\begin{proof}
This is immediate from Theorem~\ref{L:linear} and the monotonicity result,
Theorem~\ref{mono-n}.
\end{proof}

\begin{figure}[b!]
\centerline{\includegraphics[width=\textwidth]{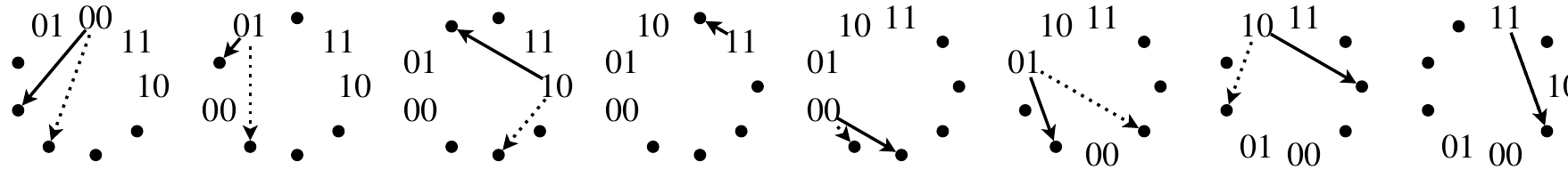}}
\vspace{24pt}

\centerline{\includegraphics[width=\textwidth]{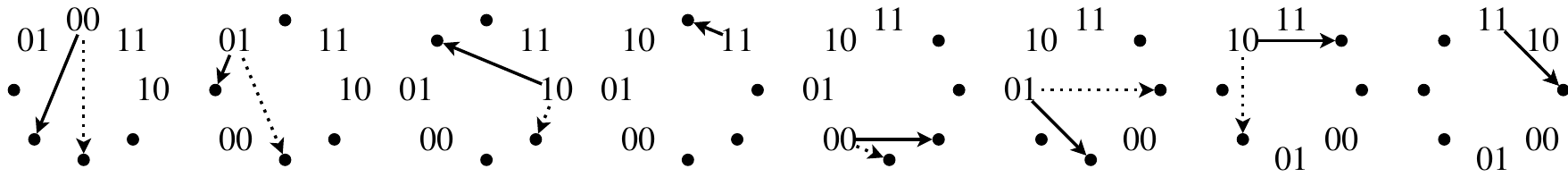}}
\caption{Avoidance coupling of $k=2^d$ walkers ($d=2$) on $K_{2^{d+1}+1}$ (upper panel) and $K^*_{2^{d+1}}$ (lower panel).  The walkers are naturally indexed by the vertices of a hypercube, while the states are naturally indexed by the cycle.  Walker $0$ and walkers $2^i$ each flip a coin to randomly choose from among two states, shown by the arrows, while the motions of the other walkers are determined by the motions of these walkers.  In each round there are a total of $d+1$ coin flips, which is the minimum amount of randomness required for a random walk on $K_{2^{d+1}+1}$ or $K^*_{2^{d+1}}$.  The avoidance coupling on $K^*_{2^{d+1}+1}$ is similar to the coupling on $K_{2^{d+1}+1}$ except that the walkers sometimes stay in place, and when they do, they stay in waves.}
\label{fig:linear}
\end{figure}

\begin{proof}[Proof of Theorem~\ref{L:linear}]
  We begin with the case of $2^d$ walkers, labeled $0,\dots,2^d{-}1 =
  \omega$ on $K_n^*$ for $n=2^{d+1}$ or $n=2^{d+1}+1$. Let $\eps^{(i)}_t$ be
  independent uniform $\{\pm1\}$ random variables. Let $\delta_t$ be
  independent uniform $\{0,1\}$ random variables. Let the trajectory of walker $j$ be
  denoted $\{X^{(j)}_t\}$.

  Let $\sum_{i=0}^{d-1} j_i2^i$ be the binary representation of $j$, for $0 \le j \le \omega$.
  Given $X^{(0)}_t$, the positions of the other walkers are given by
\[
X^{(j)}_t = X^{(0)}_t + \sum_i j_i \eps^{(i)}_t 2^i
\]
where all positions are understood modulo $n$.  We then define inductively
\[
X^{(0)}_{t+1} = X^{(\omega)}_t + 2^d + \delta_{t+1}~.
\]

  We first show that this is indeed a coupling of random walkers.  For all $j$, we have
  $X^{(\omega)}_t = X^{(j)}_t + \sum (1{-}j_i) \eps^{(i)}_t 2^i$ and so
  \begin{equation} \label{eq:Xjt}
  X^{(j)}_{t+1} - X^{(j)}_t = 2^d + \delta_{t+1} + \sum_i \bigl[(1{-}j_i) \eps^{(i)}_t + j_i \eps^{(i)}_{t+1}\bigr] 2^i.
  \end{equation}
  Since $(1{-}j_i) \eps^{(i)}_t + j_i \eps^{(i)}_{t+1}$ is $\pm1$ we find that
  the sum is uniform on the odd numbers in $[-2^d,2^d]$, and so
  $X^{(j)}_{t+1} - X^{(j)}_t$ is uniform on $[1,2^{d+1}]$, as needed for the
  walk on $K_{2^{d+1}}^*$ or on $K_{2^{d+1}+1}$.  The process
  $(X^{(j)}_t)$ is Markov since the $\eps$'s and $\delta$'s used to define
  $X^{(j)}_{t+1}$ in terms of $X^{(j)}_{t}$ in~\eqref{eq:Xjt} are disjoint
  from those used in any other time step.  Note that the $j$th
  trajectory determines all the bits, so this is also a minimum
  entropy coupling.

  Next we establish avoidance. Let $j<j'$ be two walkers. We have
  \[
  \Delta := X^{(j')}_t - X^{(j)}_t = \sum_i (j'_i-j_i) \eps^{(i)}_t 2^i.
  \]
  Note that $|j'_i-j_i|\leq1$, hence $|\Delta| < 2^d$, and so
  $\Delta =0\mod n$ implies $\Delta=0$. If $i_0$ is the minimal index such that
  $j_{i_0}\neq j'_{i_0}$ then $\Delta$ is divisible by $2^{i_0}$ but not by
  $2^{i_0+1}$, and so is non-zero. Thus there are no collisions within any
  round. Between consecutive rounds we have
  \begin{equation}
    \label{eq:delta_sum}
    \Delta := X^{(j)}_{t+1} - X^{(j')}_t = 2^d+\delta_{t+1} + \sum_i
    \bigl[j_i\eps^{(i)}_{t+1} - (1{-}j'_i)\eps^{(i)}_t\bigr] 2^i.
  \end{equation}
  Let $i_1$ be maximal such that $j_{i_1}\neq j'_{i_1}$. Since $j<j'$ this
  implies $j_{i_1}=0$ and $j'_{i_1}=1$. We have
  \[
  \bigl|j_i\eps^{(i)}_{t+1} - (1{-}j'_i)\eps^{(i)}_t\bigr| \leq \begin{cases}
    1 & i>i_1, \\ 0 & i=i_1, \\ 2 & i<i_1. \end{cases}
  \]
  Terms for $i>i_1$ contribute at most $2^d-2^{i_1+1}$ in absolute value to the sum in
  \eqref{eq:delta_sum}, while terms for $i<i_1$ contribute at most
  $2(2^{i_1}-1)$. Thus
  \[
  \Delta \in [2+\delta_{t+1},2^{d+1}-2+\delta_{t+1}]
  \]
  and so $\Delta\neq0 \mod n$.

  To see that this coupling is Markovian, note that $X^{(0)}_t$ is determined by
  $X^{(\omega)}_{t-1}$ and $\delta_t$. Similarly, $X^{(2^i)}_t$ is determined by
  $X^{(0)}_t$ and $\eps_t^{(i)}$, and the position of any other walker $X^{(j)}_t$
  (i.e., for $j$ not a power of $2$) is determined by the
  positions in that round of walkers with smaller index.

We can reduce the number of walkers to any value between $2$ and $2^d$ by simply
removing walkers other than $0$ and $\omega$.  The Markovian property is
preserved if we first remove walkers whose indices are not powers of $2$.

Finally we turn to the case of $k$ walkers on $K_n^*$ for $n=2^{d+1}+1$. To
do this we simply add to the coupling on $K_n$ rounds in which all walkers
rest, beginning with walker 0.  For the Markovian property, we need to ensure
that each walker $j \not= 0$ can detect when walker 0 has decided to rest.
This is so because on $K_n$, given $X^{(\omega)}_t$, no vertex is a possible
value for both $X^{(0)}_t$ and $X^{(0)}_{t+1}$ (otherwise $X^{(0)}_{t+1}$, which depends
only on $X_t^{(\omega)}$ and $\delta_{t+1}$ but not on $X^{(0)}_t$, might stay in
place).
\end{proof}

We say that an avoidance coupling of $k$ walkers \df{stays in waves} if, for
some distinguished walker~$w$, whenever $w$ stays in place, all the other
walkers do likewise at the following $k{-}1$ turns, while if $w$ moves, all
others do so too. (The coupling on $K_{2^{d+1}+1}^*$ in the last proof stays
in waves.)  Note that any Markovian avoidance coupling that stays in waves on
$K_n^*$ may be modified to obtain a Markovian avoidance coupling on $K_n$ by
removing all the looping rounds.

\section{Many walkers for general \texorpdfstring{\boldmath$n$}{n}}

\begin{theorem}\label{T:almost_linear}
  There exists a Markovian avoidance coupling of $k$ walkers on $K_n^*$ for any $k \leq
  n/(8\log_2 n)$, and on $K_n$ for any $k \leq n/(56\log_2 n)$.
\end{theorem}

The constants in this theorem can easily be improved. However, as noted
below, our methods will not go beyond $n/(\log_2 n)$.  To prove the theorem,
we make use of two lemmas which allow us to combine avoidance couplings.

\begin{lemma}\label{lemma:product}
  Suppose that we have avoidance couplings of $r$ walkers on $K_m^*$
  and of $s$~walkers on $K_n^*$.  Then there is an avoidance coupling
  of $k$ walkers on $K_{m n}^*$, for any $k$ satisfying $r+s-1\leq k\leq r s$.
  If the given couplings are Markovian,
  then so is the new coupling.  If the given couplings stay in waves,
  then so does the new coupling.  If the given couplings are
  minimum-entropy, then the new coupling is too.
\end{lemma}

\begin{proof}
  We identify $K^*_{mn}$ with $K^*_m\times K^*_n$, and note that if
  $X_t$ and $Y_t$ are independent random walks on $K^*_m$ and $K^*_n$
  respectively, then $(X_t,Y_t)$ is a random walk on $K^*_{mn}$.
  Given an avoidance coupling $\{X^{(i)}_t\}$ of $r$ walkers on $K^*_m$
  and an independent avoidance coupling $\{Y^{(j)}_t\}$ of $s$ walkers on
  $K^*_n$, we construct a coupling on $K^*_{mn}$ of $r s$ walkers with
  labels $(i,j)$, for $1\leq i\leq r$ and $1\leq j\leq s$.  The
  walkers move in lexicographic order.  The trajectory of walker
  $(i,j)$ is given by $(X^{(i)}_t,Y^{(j)}_t)$, which as noted above is a
  random walk on $K_{mn}^*$.  That the walkers avoid collisions follows
  from the product construction and the collision avoidance of the
  given couplings.  If the given couplings are Markovian, then since
  the walkers on $K_{mn}^*$ move in lexicographic order, the resulting
  coupling is also Markovian.  It is clear that the coupling stays in
  waves provided both original couplings do.  Finally, no randomness
  is required beyond that in the couplings on $K_m^*$ and on $K_n^*$,
  so if they are minimum entropy, so is the resulting coupling.

  To construct a coupling of fewer walkers, just eliminate some of the
  walkers, as long as walkers $(i,1)$ and $(1,j)$ (for each $1\leq
  i\leq r$ and $1\leq j\leq s$) are kept. All other trajectories are
  determined by those, so the Markov property is maintained.
\end{proof}

We remark that a variant of the above construction can be used to combine an
avoidance coupling of $r$ walkers on $K_m$ and an avoidance coupling of $s$
walkers on $K_n^*$ that stays in waves to produce an avoidance coupling of
$rs$ walkers on $K_{mn}$.

\begin{lemma}\label{lemma:m+n}
  Suppose we have avoidance couplings for $k$ walkers on $K_m$ and on
  $K_n$ (respectively, $K_m^*$ and $K_n^*$). Then we have an avoidance
  coupling for $k$ walkers on $K_{m+n}$ (respectively, $K_{m+n}^*$).
  If the original couplings are Markovian then so is the resulting
  coupling.
\end{lemma}

\begin{proof}
  Partition the vertex set of $K_{m+n}$ into two clusters $U$ and $V$
  of sizes $m$ and $n$, respectively.  We will ensure that when it is
  the first walker's turn, all of the walkers are in the same cluster.
  At each of her turns, the first walker flips an appropriately biased
  coin to decide whether to move within her current cluster or to
  switch to the other cluster. If she stays in her current cluster she
  moves according to that cluster's coupling rules, and so do the rest
  of the walkers.  If she switches to the other cluster, she moves to
  a uniformly random vertex therein.  Each subsequent walker now
  chooses a random $k$-walker configuration in the new cluster (say,
  $V$) consistent with the walkers that are already in $V$, in
  accordance with the stationary distribution on configurations of the
  $K_n$ (or $K_n^*$) coupling arising just before a move of the first
  walker.  He then moves to his allotted space in this configuration.
\end{proof}

\begin{proof}[Proof of Theorem~\ref{T:almost_linear}]
  We begin with the case of $K_n^*$.  By Theorem~\ref{L:linear} and
  Lemma~\ref{lemma:product} we have
  a Markovian avoidance coupling for $k\leq 2^d$ walkers on $K^*_n$ where
  $n$ is of the form $n=2^{a+1}(2^{d-a+1}{+}1)=2^{d+2}+2^{a+1}$ for any $a\leq d$, as well as
  $n=2^{d+1}{+}1$, $2^{d+2}{+}1$, and $2^{d+1}$.

  For general $n>0$, we write $n=\sum_i n_i 2^i$ where $n_i \in \{0,1\}$.  We define $r$ by
  \[
  n = \sum_{i=0}^d n_i (2^{d+2}+2^i) + r,
  \]
  and it is clear that $2^{d+1} | r$. If $r\ge0$, then
  Lemma~\ref{lemma:m+n} provides a Markovian avoidance coupling for
  $K_n^*$; this inequality indeed holds whenever
  \[
  n \ge \sum_{i=0}^d (2^{d+2}+2^i) = (d{+}{\textstyle\frac32})2^{d+2}-1~.
  \]

  Now any $n\geq 8$ satisfies $(d{+}2)2^{d{+}2} \le n < (d{+}3)
  2^{d{+}3}$ for some integer $d\geq 0$. Thus the above gives a
  Markovian avoidance coupling for any number of walkers up to $2^d$.
  By the first inequality, $8 \times 2^d \leq n$, so $d{+}3\leq \log_2 n$,
  which combined with the second inequality gives $2^d > \frac18 n/(d{+}3)
  \geq n/(8\log_2 n)$, proving the theorem for $K_n^*$ for $n\geq 8$.
  The claim of the theorem is trivial for $n<8$.

  We now turn to the case of $K_n$ (without loops). Recall that if we have
  a Markovian avoidance coupling on $K_n^*$ that stays in waves,
  then removing the looping rounds yields such a coupling on $K_n$.
  Fix $d\geq 1$, and let $S$ be the set of values of $n$ for which Markovian
  avoidance couplings exist on $K_n^*$ for \textit{every\/} number of
  walkers up to $2^{2d-1}$, all of them staying in waves. By
  Lemma~\ref{lemma:m+n}, $S$ is closed under addition.
  From Theorem~\ref{L:linear}, we see that $S$ contains $2^c{+}1$ for all
  $c\ge 2d$.  Using Theorem~\ref{L:linear} and Lemma~\ref{lemma:product},
  when $a\geq1$, $b\geq1$, and $a+b\ge 2d{+}1$, there is a Markovian avoidance
  coupling for $n=(2^a{+}1)(2^b{+}1)$ with $x\,2^{b{-}1}-y$ walkers,
  where $1\leq x\leq 2^{a{-}1}$ and $0\leq y<2^{b{-}1}$.  In particular,
  $S$ contains $(2^a{+}1)(2^b{+}1)$, and specifically $S$ contains
  $2^{2d+1}+1+2^i+2^{2d+1-i}$ for all $1\leq i\leq d$ (and also for $i=0$
  using Theorem~\ref{L:linear} and Lemma~\ref{lemma:m+n}).

  For any $m<2^{d+1}$ we write $m = \sum_{i\le d} m_i 2^i$ with $m_i \in \{0,1\}$, and denote by $\hat{m} =
  \sum m_i 2^{d-i}$ the number with reversed binary expansion.  Then for $m\neq0$, $S$ contains
  \[
  \sum_{i\le d} m_i (2^{2d+1}+1+2^i+2^{2d+1-i}) =
  \|m\| (2^{2d+1}{+}1) + m + 2^{d+1}\hat{m},
  \]
where $\|m\|:=\sum m_i$ denotes the Hamming weight of $m$.  For simplicity
(at the expense of the final constant) we eliminate the dependence on Hamming
weight: since $\|m\|\leq d{+}1$ and $2^{2d+1}{+}1\in S$ we have
\begin{equation}\label{eq:thisinS}
    (d{+}1)(2^{2d+1}{+}1) + m + 2^{d+1}\hat{m} \in S
\end{equation}
  (which holds also for $m=0$).
  In the same way, but using
  $2^{2d+2}+1+2^i+2^{2d+2-i}$ instead, we find that
  \begin{equation}\label{eq:thatinS}
    (d{+}1)(2^{2d+2}{+}1) + m + 2^{d+2}\hat{m} \in S.
  \end{equation}

Write $m' = 2^{d+1}{-}1{-}m =\sum (1{-}m_i) 2^i$, and observe that $\hat{m'}
= \hat{m}'$. Using \eqref{eq:thisinS}, together with~\eqref{eq:thatinS} with
$m'$ in place of $m$, and adding, we get
  $k_0 + 2^{d+1}\hat{m'} \in S$ where $k_0 = (3d{+}5)2^{2d+1} + 2d+1$.
Adding another copy of \eqref{eq:thisinS} we find that
 $k_1{+}m \in S$, where $k_1 = (d{+}2)2^{2d+3} - 2^{d+1} + 3d+2$.
Since the last two statements hold for all values of $m<2^{d+1}$, we may
combine them to deduce, for any $m_0,m_1<2^{d+1}$, that
\[
k_2 + 2^{d+1}m_1 + m_0 \in S,
\]
where
\[
k_2=k_0+k_1 = (7d{+}13)2^{2d+1}-2^{d+1}+5d+3.
\]

It follows that $[k_2,k_2{+}2^{2d}]\subseteq S$, and since
$2^{2d}{+}1\in S$, any integer at least $k_2$ is in $S$. Thus for any
$n\ge 7(d{+}2)2^{2d+1}$, there is a Markovian avoidance coupling for
any number up to $2^{2d-1}$ walkers on $K_n$.
Given $n$, choose $d$ so that $7(d{+}2)2^{2d+1} \le n < 7(d{+}3)
2^{2d+3}$.  From the second inequality we have $2^{2d-1} >
\frac{n}{7\times 16(d{+}3)}$.  From the first inequality we have $2(d{+}3)
\leq \log_2 n - \log_2(7(d{+}2)) + 5 \leq \log_2 n$ (provided $d\geq3$).
But $d\geq3$ for any $n\geq 7(3{+}2)2^{2\times3+1} = 4480$.  So for
$n\geq 4480$ we can couple up to $\frac{n}{7\times 8\log_2 n}$ walkers
on $K_n$.

Since there exists a Markovian avoidance coupling of $8$ walkers on
$K_{17}$ and on $K_{33}$, such a coupling also exists on $K_n$ for any
nonzero $n=17a+33b$ with $a,b\geq0$. This includes all $n>
511=33\times17-33-17$, and implies the claim for $512\leq n\leq 4480$
(since $\frac{4480}{56\log_2 4480} < 8$).  Finally, the claim is
trivial for $n\leq 511$ since $\frac{n}{56\log_2 n} < 2$.
\end{proof}

We combined the number-theoretic avoidance coupling from
Theorem~\ref{L:linear} with the sum and product lemmas to obtain an
avoidance coupling with $\Omega(n/\log n)$ walkers for any $n$.  Given
these three ingredients, this general-$n$ construction is in a sense
best possible up to constants.  More precisely, we argue below that
these three ingredients cannot be combined to obtain a coupling of
more than $n/\|n\|$ walkers on $K_n$ or $K_n^*$, where $\|n\|$ is the
Hamming weight of $n$.

By the distributive law, any coupling that can be constructed using
the sum and product lemmas \ref{lemma:m+n} and \ref{lemma:product}
can be done by taking sums of products of
basic constructions.  Consider the product of $s$ basic couplings of
$2^{d_j}$ walkers on either $2^{d_j+1}+1$ or $2^{d_j+1}$ vertices.
Note that $\|ab\| \leq \|a\|\|b\|$.  The product lemma gives a
coupling of $2^d$ walkers on $K_m$ or $K_m^*$, where $d=\sum d_j$ and
$m\geq 2^{d+s}$ and $\|m\| \leq 2^s$.  In particular $m \geq 2^d \|m\|$.
Next suppose that $n$ is the sum of several such product terms, say $n=\sum_i
m_i$, each corresponding to the same $d$.  Then $n \geq 2^d \sum_i
\|m_i\| \geq 2^d \|n\|$.  In particular the number of walkers is at
most $n/\|n\|$.

Thus, to improve on the $\Omega(n/\log n)$ bound for general $n$, more
ingredients would be needed.

\section{Negative result}

In the negative direction, we have very little.
\begin{theorem}
  No avoidance coupling is possible for $n{-}1$ walkers on $K_n^*$, for $n\geq 4$.
\end{theorem}

\begin{proof}
We exploit the effect that it is difficult for a walker to leave a vertex $v$
at one step and then immediately return to $v$ at the next step. This
requires that none of the other walkers enter $v$ in the interim.
But since $v$ is the only available vertex for a move, this means that all
other walkers must remain stationary.

Let $A_t^i$ be the event that the $i$th walker is in the same position
at times $t{-}1$ and $t{+}1$, but in a different position at time $t$.
Let $B_t^i$ be the event that the $i$th walker is in the same position at times
$t{-}1$ and $t$.

Suppose an avoidance coupling exists.
From the observation in the first paragraph, the events $A_t^1$ and $A_t^2$
are disjoint, and each of them implies the event $B_t^3$. Since each
walker individually performs a random walk, the events $A_t^1$ and $A_t^2$
have probability $(n{-}1)/n^2$, so that the probability of $B_t^3$
must be at least $2(n{-}1)/n^2$. But the probability
of $B_t^3$ should be exactly $1/n$, which is less than $2(n{-}1)/n^2$.
This gives a contradiction, as required.
\end{proof}

\pagebreak
\section{Open Problems}

We have barely scratched the surface of avoidance coupling in this work; in
particular we have considered only complete graphs and concentrated on
discrete, alternating, Markovian couplings.  Even in this limited realm, many
intriguing open questions remain:

\begin{enumerate}

\item\textbf{Maximum number of walkers.}  Is there an avoidance coupling for
    a linear number of walkers on the unlooped complete graph $K_n$ for
    general $n$? Can upper bounds of the form $cn$ for $c<1$ be found for
    the maximum number of walkers that can be avoidance-coupled on $K_n$
    or $K_n^*$? Ditto for Markovian couplings?

\item\textbf{Monotonicity in \boldmath$n$.} If there is an avoidance coupling for
  $k$ walkers on $K_n$, must there necessarily be one for $k$ walkers
  on $K_{n+1}$?  Similarly in the Markovian case, for either $K_n$
  versus $K_{n+1}$ or $K_n^*$ versus $K_{n+1}^*$.

\item\textbf{Monotonicity in \boldmath$k$.} If there is a Markovian avoidance coupling
  for $k$ walkers on $K_n$ (or $K_n^*$), is there one for $k{-}1$ walkers
  on the same graph?  The answer is ``yes'' for non-Markovian couplings,
  since the $k$th walker can be imagined.  The answer is ``yes'' for the
  Markovian couplings that we exhibited, but it is not clear if this holds
  in general.

\item\textbf{Monotonicity in loop weights.} Suppose that $K_n$ is equipped
  with loops of weight $w$, so that a walker loops with probability
  $w/(w{+}n{-}1)$.  If there is an avoidance coupling for $k$ walkers on
  $K_n$ with loops of weight $w$, must there be one with loops of weight
  $w'$, where $w' > w$?  The answer is ``yes'' for non-Markovian couplings,
  but what about the Markovian case?  In particular, is the maximum number
  of Markovian avoiding walkers always at least as great on $K_n^*$ as it
  is on $K_n$?

\item\textbf{Minimum entropy couplings.}  Does existence of an avoidance
    coupling imply existence of a stationary avoidance coupling whose
    entropy equals that of a single random walk?

\item\textbf{1-avoidance.} What is the largest $p$ for which $k$ i.i.d.\
    $\bern(p)$ sequences can be coupled, taking turns to move as usual,
    so that no two simultaneously take the value $1$?
\end{enumerate}

\noindent
\textbf{Acknowledgments.}
The authors are grateful to the Institute for Elementary Studies and
to the Banff International Research Station, where this work began,
and to Microsoft Research.  Supported in part by NSERC (OA), an EPSRC
Advanced Fellowship (JM), and NSF grant \#0901475 (PW).


\end{document}